\documentclass[12pt]{amsart}
\usepackage[utf8]{inputenc}

\usepackage{lineno}

\usepackage[a4paper, top=3cm, left=2.5cm, right=2.5cm, bottom=3cm]{geometry}

\usepackage{hyperref}
\hypersetup{
    colorlinks=true,
    linkcolor=magenta,
    filecolor=magenta,
    urlcolor=blue,
    citecolor=blue,
}

\usepackage{amsmath}
\usepackage{amsfonts,amsthm, pb-diagram}
\usepackage{amssymb,comment}
\usepackage{graphicx,color}
\usepackage{enumerate}
\usepackage[normalem]{ulem}
\usepackage{thmtools,cleveref}
\usepackage{mathtools}
\usepackage{bm}
\usepackage{mathtools}

\usepackage{natbib}
\setcitestyle{square}

\renewcommand{\models}{\vDash}

\newcommand{\ran}{\operatorname{ran}}

\newcommand{\dom}{\operatorname{dom}}

\newcommand{\thzfc}{\mathrm{ZFC}}
\newcommand{\rulec}{\mathrm{CR}}

\newcommand{\acal}{\mathcal{A}}

\newcommand{\restrictto}[1]{{\upharpoonright}_{#1}}

\usepackage{xcolor}
\usepackage{listings}
\usepackage{color}
 \usepackage[draft]{minted}

\newcommand{\val}{\mathrm{val}}
\newcommand{\cou}{\mathrm{count}}

\definecolor{codebg}{HTML}{ffffff}
\definecolor{dkgreen}{rgb}{0,0.6,0}
\definecolor{gray}{rgb}{0.5,0.5,0.5}
\definecolor{mauve}{rgb}{0.58,0,0.82}

\newcommand{\highlight}[1]{\textbf{\textit{#1}}}
\newcommand{\varphivalue}{\mathrm{val}_{A}(\varphi)[\sigma]}
\newcommand{\psivalue}{\mathrm{val}_{A}(\psi)[\sigma]}
\newcommand{\chivalue}{\mathrm{val}_{A}(\chi)[\sigma]}
\newcommand{\varphivaluep}{\mathrm{val}_{A}(\varphi)[\tau]}

\definecolor{ao(english)}{rgb}{0.0, 0.5, 0.0}

\begin{document}

\title{Which came first, set theory or logic?}

\author{J. Julian Pulgarín}
\address{Senior Software Engineer at Block, Inc., Salt Lake City, Utah, USA}
\email{\href{mailto:julian@pulgarin.co}{julian@pulgarin.co}}
\urladdr{\url{https://www.researchgate.net/profile/Julian-Pulgarin}}

\author{Andr\'es F. Uribe-Zapata}
\address{TU Wien, Faculty of Mathematics and Geoinformation, Institute of Discrete Mathematics and Geometry, Wiedner Hauptstrasse 8--10, A--1040 Vienna, Austria }
\email{andres.zapata@tuwien.ac.at} 
\urladdr{\url{https://sites.google.com/view/andres-uribe-afuz/home}}

\thanks{ffff}



\date{\today}

\makeatletter
\def\@roman#1{\romannumeral #1}
\makeatother

\newcounter{enuAlph}
\renewcommand{\theenuAlph}{\Alph{enuAlph}}

\numberwithin{equation}{section}
\renewcommand{\theequation}{\thesection.\arabic{equation}}

\theoremstyle{plain}
  \newtheorem{theorem}[equation]{Theorem}
  \newtheorem{corollary}[equation]{Corollary}
  \newtheorem{lemma}[equation]{Lemma}
  \newtheorem{mainlemma}[equation]{Main Lemma}
  \newtheorem{fact}[equation]{Fact}
  \newtheorem{prop}[equation]{Proposition}
  \newtheorem{claim}[equation]{Claim}
  \newtheorem{question}[equation]{Question}
  \newtheorem{problem}[equation]{Problem}
  \newtheorem{conjecture}[equation]{Conjecture}
  \newtheorem*{theorem*}{Theorem}
  \newtheorem*{mainthm*}{Main Theorem}
  \newtheorem{teorema}[enuAlph]{Theorem}
  \newtheorem*{corollary*}{Corollary}
\theoremstyle{definition}
  \newtheorem{definition}[equation]{Definition}
  \newtheorem{rdefinition}[equation]{Recursive Definition}
  \newtheorem{ruler}[equation]{Construction Rule}
  \newtheorem{example}[equation]{Example}
  \newtheorem{remark}[equation]{Remark}
  \newtheorem{notation}[equation]{Notation}
  \newtheorem{context}[equation]{Context}
  \newtheorem{observation}[equation]{Observation}
  \newtheorem*{definition*}{Definition}
  \newtheorem*{acknowledgements*}{Acknowledgements}
  \newtheorem*{funding*}{Funding}

\def\sectionautorefname{Section}
\def\subsectionautorefname{Subsection}

\begin{abstract}
    The construction of first-order logic and set theory gives rise to apparent circularities of mutual dependence, making it unclear which can act as a self-contained starting point in the foundation of mathematics. In this paper, we carry out a metatheoretic and finitistic development of a first-order logical system in which we define formal set theory (ZFC). We contend that the techniques employed in constructing this system offer a philosophically sound basis for introducing ZFC and proceeding with the conventional formal development of mathematics. Lastly, by examining the chronology of the aforementioned construction, we attempt to answer the titular question.
\end{abstract}

\maketitle

\section{Introduction}\label{sec:intro}

“Zermelo–Fraenkel set theory is the foundation of mathematics”, although not without controversy, is a common assertion in virtually all mathematical contexts. From a formal standpoint, $\thzfc$ is a first-order theory; thus, a system of first-order logic must be established prior to its development. At the same time, during the development of this system, we require the use of certain tools from set theory. For example, in the proof of Gödel's completeness theorem, we make use of equivalence relations and the axiom of choice in the form of Zorn's lemma. An even more fundamental example is the proof of the deduction theorem, which requires the application of induction over the length of a formula. There is therefore an apparent circularity within the construction of logic and set theory—a worrying prospect for the rest of mathematics that rests upon them. In searching for a solution to this paradox, a natural question arises: \highlight{which came first, set theory or logic?}

Although this question arises early in a mathematician's journey, the literature addressing it is relatively sparse and lacks comprehensive treatment (see, for example, \cite{Kleene})—possibly because of the philosophical questions that quickly arise. In this paper, we will tackle the paradox by following an outline given by Kunen, in his textbook, \textit{The Foundations of Mathematics}:

\begin{quotation}
    We start off by working in the metatheory [...]. As usual, the metatheory is completely finitistic, and hence presumably beyond reproach. In the metatheory, we develop formal logic, including the notion of formal proof. We must also prove some (finitistic) theorems about our notion of formal proof to make sense of it. This includes the Soundness of our proof method [...], rephrased to say that if $\Sigma \models \varphi$ then $\varphi$ is true in all \textit{finite} models of $\Sigma$ \citeyearpar[III.2]{Kunen}.
\end{quotation}

The metatheory here mentioned by Kunen refers to the unformalized mode of reasoning that humans find “intuitively convincing” \cite[\S15]{Kleene} by virtue of our rational faculties. Relying on our metatheory, instead of ZFC, to construct our first-order logical system, will be the key in removing all suspicions of circularity. Kunen's suggestion to include a proof of the soundness theorem is not accidental: it assures us that we can trust the demonstrations of our system.

Once constructed, we will develop ZFC within our logical system. Although not done in this paper, with the machinery of ZFC and logic in place, one can then construct logic a \textit{second} time, but now with the ability to use the completed infinite, and tools like Zorn's lemma and the deduction theorem, paradox-free.

Kunen's proposal lacks detail, and it is not clear, for example, what is the precise nature of the metatheory, or of the universe of objects it is used to talk about. Similarly, it is not explicit about what results we need in order to be able to prove the soundness theorem, without walking into a circularity or using non-metatheoretical reasoning. The goal of this paper is to make all these things explicit.

We will structure our paper in the following manner: in the second section we establish what we will understand as the metatheory, the types of reasoning it permits, and we justify the definitions and recursive arguments that are used throughout. In the third section, we construct a finitistic set theory founded on the base intuitions we have about how we think about and manipulate collections of objects in the physical world. In the fourth section, we use these tools to construct a logical system capable of hosting ZFC, with the soundness theorem marking our finish line. These series of constructions clearly show how disastrous circularity is averted, as well as giving us a concrete model that allows us to philosophize on the chronology of set theory and logic in the fifth and last section.

\section{Metatheoretical Logic: An Intuitionistic Approach}\label{sec-2}

The first task before us is to decide what types of reasoning we will admit during our construction of first-order logic. What our metatheory looks like is constrained by various factors. First, it must be powerful enough to prove all of the results that we need to successfully build our system of formal logic. Secondly, it must consist of modes of reasoning that would be accepted by even the most ardent skeptic, for our whole purpose would be thwarted if someone could credibly claim that we are reasoning in ways that must themselves be justified by appealing to an even more basic metatheory.

In this paper, we will use intuitionistic logic as our metatheory, as we feel it is particularly well-suited given the two constraints just mentioned. Intuitionistic logic refers to the logical principles underlying intuitionistic mathematics, which is the subset of mathematics that aims to give primacy to human intuition and reduce our reliance on ontological assumptions. It should come as no surprise then that this logic can function as the ideal metatheory when trying to prove things from first principles. We will not give a full account of intuitionistic logic here but will highlight many of its properties that we will use in this paper. To read more about intuitionistic logic, see \cite[Ch.~5]{phil}.

The foundation of our metatheory is bootstrapped from our intuitive notions of self-evident logical facts and finite sets. For example, we assume that our rational faculties allow us to distinguish between different symbols and compare quantities. This means that we are able to assert, without reference to any outside axiomatic system or justification, statements like ``$a$ equals $b$ implies $b$ equals $a$" and ``$x$ equals $x$". Similarly, using our intuitions about how quantities work in the real world, we are also able to assert statements about natural numbers and their arithmetical relations.

An observant reader might argue that in the preceding paragraph, we have surreptitiously introduced set theory into our metatheory, by relying on facts about finite \textit{sets}. To address this objection, we put forth two points. Firstly, our reliance is only on set-theoretic facts pertaining to the comparison of sizes of finite sets. Secondly, we do not need to accept any of these facts on faith; each can be empirically verified. For instance, suppose we assert that set $\mathcal{A}$ contains more elements than set $\mathcal{B}$. This claim can be verified using a simple pen-and-paper method: make a mark at the top of the page for each element in $\mathcal{A}$, and at the bottom of the page for each element in $\mathcal{B}$. Next, begin crossing off marks, one-by-one from both the top and the bottom. If all the marks at the bottom are exhausted while some remain at the top, it verifies that set $\mathcal{A}$ indeed has more elements. The only thing you need to take on faith is that humans, in general, are able to follow the above instructions by virtue of our sense perception and basic rationality. The method outlined above can be conceptualized as a basic algorithm: an explicit procedure that can be mechanically followed to arrive at a result. In the next subsection, we show how we will apply similarly structured algorithms to prove more complex statements in our metatheory.

\subsection{Proofs as Algorithms}
  
Proofs in intuitionistic logic are very similar to algorithms in computer programming. A proof of the statement $\varphi$ will be an algorithm that terminates in finite time, and which will serve to convince us of the truth-value of what is asserted by $\varphi$. Intuitionistic logic makes rigorous what an algorithm must look like to prove a specific statement, but roughly the algorithm should exhaustively verify all that is asserted by the statement. Since the algorithm must terminate after a finite amount of time, in principle a human could execute the algorithm on a computer (or pen and paper, if she were inclined), and thus be convinced of the truth of the proposition that it is a proof for. For example, we could envision a program that proves the statement, \lq$7$ is prime\rq, by checking that no integer strictly between $1$ and $7$ is a divisor of $7$. Algorithms are inherently precise and capable of being executed across various platforms without requiring originality or creativity, and their structure and content lend themselves well to rational analysis. These attributes, among others, render them particularly apt as the basis of intuitionistic proof theory.

\subsection{Recursive Definitions and Proofs}

Recursive programming is a well-known technique in computer programming that allows one to solve problems by breaking them down into smaller instances of the same problem. We will employ this technique when  proving theorems about objects that have a recursive structure. Logical formulas are an example of just such an object: they are composed of simpler subformulas. If we wanted to prove that a formula $\varphi$ has a certain property, we could use a recursive algorithm that would prove the property for the subformulas of $\varphi$, and for the subformulas of the subformulas, etc., until reaching the basic atoms of the formula which are not themselves decomposable. Since formulas are of finite length, this algorithm will eventually terminate and therefore serves as an exhaustive proof that $\varphi$ has the property in question. An algorithm of this type can be thought of as a proof using mathematical induction, and in fact, all of our recursive proofs will look like regular inductive proofs, instead of explicit computer programs.

We will use recursive algorithms not only to prove theorems about recursive objects but also to define them. These definitions can be thought of as algorithms that can detect whether a given object satisfies the recursive definition in question. As an example, consider the following recursive definition:

\begin{rdefinition}
    A finite sequence of symbols $s$ is a \highlight{p-sequence} if, and only if, one of the following holds:
    
    \begin{enumerate}
        \item $s$ is an empty sequence.
        \item $s$ is of the form ($s'$), where $s'$ is a p-sequence.
    \end{enumerate}
\end{rdefinition}

From this definition, it is clear that $s$ is a p-sequence if and only if it is composed of $n$ open parentheses followed by $n$ closed parentheses. Using Python, the algorithm induced by the recursive definition would look as follows:

\begin{minted}
[
frame=lines,
framesep=2mm,
baselinestretch=1.2,
bgcolor=codebg!20,
fontsize=\footnotesize,
linenos
]
{python}
def is_p_sequence(s):
    if not s:  # s is an empty sequence
        return True

    return (
        s[0] == '(' and  # s begins with (
        s[-1] == ')' and  # s ends with )
        is_p_sequence(s[1:-1])  # s' is a p-sequence
    )
\end{minted}

\subsection{Universal Quantification}

At first glance, it may seem like we are severely limited in our ability to reason when using our metatheory. Using the tools discussed thus far there does not seem to be a way to prove things about (potentially) infinite collections of things, such as being able to prove that all integers have a unique prime factorization. Fortunately, there is a way to interpret statements of this form in our metatheory, without dramatically changing our definition of proof. For us, proving a statement of the form, ``for all x, $\varphi (x)$", will be done by exhibiting a \textit{recipe}—an algorithm, to be precise—that transforms a particular element $x$, into an algorithm proving $\varphi (x)$. In other words, our proof will be a \textit{guarantee} that, for any item $x$ in the collection, we can produce a valid intuitionistic proof of $\varphi (x)$. The difference here is mostly philosophical: it allows us to prove things for elements in sets of unbounded size, without making ontological commitments to the completed infinite. In practice, intuitionistic universal quantification proofs look a lot like their counterparts in classical mathematics: we assume we have an arbitrary element of a given collection and proceed to prove something about this arbitrary element.

Notice that this means that the traditional inductive proof used to prove the aforementioned unique factorization of the integers will be acceptable to us since it can be used to create the desired recipe. That is, given an integer $n$, the traditional inductive proof can be used to create a recursive algorithm that will function as a proof of $n$'s unique prime factorization. We will make heavy use of universal quantification to prove that all sets, all assignments, or all formulas, have certain properties.

Closely  related  to  the  problem  of  choosing  our  manner  of  reasoning,  is  choosing  the types of objects over which we will reason about. It is clear that if our aspirations of proofs as being finitely verifiable are to be fulfilled, then the objects of study must also be finite. The nature of these objects and the ways in which we can manipulate them will be the topic of the next section.

\section{Set Metatheory: Rules for Manipulating Rational Objects}  

In this section, we will define and study finite sets. We will begin with the primary intuitions we have of collections of tangible, real-world objects. Then, we will abstract out these intuitions to rational objects—that is, to conceptual entities that exist within our minds—which will allow us to define a finite set theory that corresponds to the usual formalization. Note that all of our reasoning will be intuitionistically valid.

\subsection{Collections and Membership}

A \highlight{collection} is a finite, unordered \emph{list} of objects without repetitions. We use calligraphic letters to refer to collections: $ \mathcal{A}, \mathcal{B}, \mathcal{C}, $ etc. \highlight{Membership} is a relation that can occur between an object $ a $ and a collection $ \mathcal{A} $: if $ a $ is one of the objects listed in the collection $ \mathcal{A} $, then we say that $ a $ \highlight{belongs} to $ \mathcal{A} $ or that $ a $ is an \highlight{element} of $ \mathcal{A} $, and denote it as $ a \in \mathcal{A}. $ We also consider that collections are completely determined by their listed objects, that is:

\begin{definition}
     Two collections are \highlight{equal} if, and only if,  they have the same elements.
\end{definition}

Given two collections, $\mathcal{A}$ and $\mathcal{B}$, we can create a new collection whose objects are the objects of $\mathcal{A}$ and the objects of $\mathcal{B}$, by listing the objects of both and eliminating any repetitions. This process is known as the \highlight{union} of A and B. Similarly, collections can be intersected or subtracted.

We can now ascend to a higher level of abstraction, and consider collections themselves as \highlight{abstract objects}. This will allow us to form \highlight{abstract collections} (whose elements are abstract objects) to which we can apply the same manipulation rules we have just described. Notice that if $\mathcal{A}$ and $\mathcal{B}$ are collections, we can create a new collection consisting precisely of the abstract objects $\mathcal{A}$ and $\mathcal{B}$. Generalizing  this property we obtain what we call the \highlight{construction rule}, which we abbreviate as $\rulec$:

\begin{ruler}
          If $\acal_{0}, \dots, \acal_{n}$ are collections, then we can construct a collection whose members are exactly  $\acal_{0}, \dots, \acal_{n}. $
\end{ruler}          

The justification for the following notation is given by virtue of the equality between collections:

\begin{notation}
      Given collections $\acal_{0}, \dots, \acal_{n}$, we denote the collection whose existence is given by $\rulec$ as  $\{ \acal_{0}, \dots, \acal_{n} \}$.
\end{notation}

Note that a collection can have no elements. We denote this collection by $\emptyset$ and we call it the \highlight{empty collection}.

\subsection{Finite Sets}

Based on the intuitions mentioned in the previous subsection, and starting from the empty collection, we are going to define a category of collections that corresponds to what we typically consider as the finite sets.

\begin{rdefinition}\label{def-set}
        A collection $\mathcal{A}$ is a \highlight{finite set}, or simply, \highlight{set}, if, and only if, one of the following holds:
     \begin{enumerate}
         \item $\mathcal{A}$ = $\emptyset$.
         \item $\mathcal{A} = \{ A_{0}, \dots, A_{n} \}$, where all of $A_{0}, \dots, A_{n}$ are sets. 
     \end{enumerate}
\end{rdefinition}

For example, $ \emptyset, \ \{ \emptyset\}, \ \{\{\emptyset \} \},$ and $ \{\emptyset, \ \{\emptyset \} \} $ are sets. Notice that from the previous definition, it follows that every set can be constructed by starting from the empty set and applying $ \rulec$ finite times. We use  capital letters of the Latin alphabet to refer to sets: $A, B, C,$, etc.

Furthermore, the following result is immediate:

\begin{lemma}    
    Every element of a set is a set.
\end{lemma}

\begin{definition}
     If $\acal$ is a collection and $B$ is a set, we say that $\acal$  is a \highlight{subcollection} of $B$, denoted by $\acal \subseteq B$, if, and only if, every element of $\acal$ is an element of $B.$
\end{definition}

Notice that if $\acal \subseteq B,$ then every object of $\acal$ is a set, and by \autoref{def-set} \, (2), $\acal$ is a set. We thus obtain the following result:

\begin{lemma}    
    Every subcollection of a set is a set.
\end{lemma}

Using the definition of set it immediately  follows that the result of joining or subtracting sets is itself a set:

\begin{lemma}
    Consider sets $A$ and $B$. Then: 
    \begin{enumerate}
        \item We can construct a set whose elements are exactly the elements of $A$ and the elements of $B,$ which is denoted as $A \cup B. $
        
        \item We can construct a set whose elements are exactly the elements of $A$ that are not in $B,$ which is denoted by $ A \setminus B.$
    \end{enumerate}
\end{lemma}

\subsection{Functions and natural numbers represented as sets and sequences}

In this subsection, using the definition and properties of sets studied in the previous subsection, we introduce three notions that are fundamental in our later development: functions, natural numbers as sets, and finite sequences. We begin with the definition of ordered pair:

\begin{definition}
     If $A$ is a set and $a, b \in A,$ we define the \highlight{ordered pair} of $a$ and $b$, denoted by  $(a, b)$, as: $(a, b) \coloneqq \{ \{a\}, \{a,b \} \}.$
\end{definition}

Since in particular $a$ and $b$ are sets, it is clear by \autoref{def-set}, that $(a, b)$ is a set. On the other hand, it is also clear that if $A$ and $B$ are sets, then there are a finite number of pairs of the form $(a, b)$ where $a \in A$ and $b \in B.$ Applying \autoref{def-set} once more, it follows that the collection of all of them is a set:

\begin{notation}
    If $A$ and $B$ are sets, we denote the set which contains all elements of the form $(a, b)$, with $a \in A$ and $b \in B$, by $A \times B$. 
\end{notation}

Using ordered pairs, we can define functions in the usual way:

\begin{definition}
     Consider sets $A, B$, and $f.$ We say $f$ is a function from $A$ into $B$, denoted by $f: A \to B$, if, and only if, it satisfies the following conditions:
     \begin{enumerate}
         \item $f \subseteq A \times B.$
         
         \item If $a \in A, \ b_{1}, b_{2} \in B$ and $(a, b_{1}), (a, b_{2}) \in f,$ then $b_{1} = b_{2}.$
         
         \item For all $a \in A,$ there is an element $b \in B$, such that $(a, b) \in f.$ 
     \end{enumerate}
          
     In this case, we say $A$ is the \highlight{domain} of $f$ and we denote it by $\dom(f).$ We denote by $f(a)$ the unique element $b$ such that $(a, b) \in f$.
\end{definition}

Now, we define the range of a function:

\begin{definition}
     Consider a function $f \colon A \to B.$ An \highlight{image} of $f$ is an element $b \in B$ such that there is an $a \in A$, with $f(a) = b.$ Also, we define the \highlight{range} of $f$, denoted by $\ran(f)$, as the set of all the images of $f.$  
\end{definition}

Notice that the range of a function is a set since it is a subset of $ B. $

\begin{definition}\label{res-func}
     Assume  $f$ is a function  with $\dom(f) = A$, and let $C \subseteq A.$ We define the \highlight{restriction} of $f$ to $C,$ denoted by $f \restrictto{C}$, as: $f \, \restrictto{C} \, \coloneqq \{ (a, b) \in f \colon  a \in C\}.$ 
\end{definition}

As mentioned previously in \autoref{sec-2}, our metatheory allows us to consider natural numbers and their properties intuitively. However, in order to define the notion of a \emph{sequence}, we will now proceed to represent the natural numbers using finite sets.

\begin{rdefinition}\label{natural-nums}    
    The natural numbers are represented, recursively, as follows:

    A set $n$ is a natural number if, and only if, one of the following holds: 

    \begin{enumerate}
        \item $n = \emptyset$, which represents the number $0$;
        
        \item $n = m \cup \{ m \}$, where $m$ is a natural number, in which case $n$ represents the number $m + 1$.
    \end{enumerate}
\end{rdefinition}

\begin{lemma}
    A natural number is the set of all the numbers less than itself.
\end{lemma}

\begin{proof}
    We will prove this by induction. Note that this is trivially true for $0$.
    
    Consider $n$ and $m$, natural numbers, such that $n = m + 1$, and assume that $m$ is the set of all the numbers less than itself. The set of numbers less than $n$ is composed of all the numbers less than or equal to $m$. Note that this is $m \cup \{ m \}$, and we therefore obtain our desired result.
\end{proof}

This representation allows us to give the following definition:

\begin{definition}
     A \highlight{sequence} of elements of a set $A$ is a function $s$ such that $\dom(s)$ is a natural number, and $\ran(s) \subseteq A.$ If $\dom(s) = n,$ we say $s$ has length $n$ and we denote it by $\vert s \vert = n.$
\end{definition}

For simplicity, when we refer to sequences, we do not use function notation, but we will think of the sequences as ``words'' formed with the elements of $ A. $ For example, suppose that $A = \{a, b, c\}.$ Then,  $s = \{  (0, a), (1, b), (2, c) \}$ is a sequence of elements of $A$ of length $3$, and which we will denote simply as \textit{abc}. It is understood that the position occupied by the element in the sequence is the preimage under the  function $s$. On the other hand, the \highlight{restriction} of a sequence to a natural number $k$  is simply the sequence formed from its first $k$ ``letters''.  For example, $s \restrictto{2}$ is the sequence $\textit{ab}.$ 

Armed with the assurance that we are reasoning correctly by working within our metatheory and that our intuitive notions of sets are paradox-free, we can now begin to construct our system of first-order logic.

\section{Finite First-order Logic}\label{section4}

 Our aim in this section is to construct a system for which we can prove that it is sound—that is, show that it only proves ``true'' theorems. This will be sufficient for our purposes because once we know we are in a system that cannot lead us to falsehood, we can rederive formal logic from within the system, and trust its conclusions with the same surety as we do those of the metatheory.

\subsection{Syntax}

In this subsection, we aim to give the syntactic rules for our logical system. We begin by introducing the alphabet in which we will write our formulas.

\begin{definition}
    We will refer to the symbols $x_1, x_2, \ldots$, as \highlight{variables}. A \highlight{logical symbol}\footnote{Logical symbols can be thought of as sets, since, as done for the natural numbers in \autoref{natural-nums}, we can create a mapping between these symbols and sets that uniquely represent them. For example, we could have $x_n$ be $(1, n)$, ``$=$" be $(2, 1)$, ``$\in$" be $(3, 1)$, etc.} is either a variable or one of the following symbols: $=, \in, \neg, \vee,), (, \forall$.
\end{definition}

\begin{rdefinition}\label{def-formula}
    Given $\varphi$, a sequence of logical symbols, we say $\varphi$ is an \highlight{atomic formula} if it is of the form $x = y$ or $x \in y$, where $x$ and $y$ are variables. We say $\varphi$ is a \highlight{formula} if, and only if,  one of the following conditions holds: 
    
    \begin{enumerate}
        \item $\varphi$ is an atomic formula.
        
        \item $\varphi$ is of the form $\neg (\psi)$, where $\psi$ is a formula.
        
        \item $\varphi$ is of the form $(\psi) \vee (\chi)$, where $\psi$ and $\chi$ are formulas.
        
        \item$\varphi$ is of the form $\forall x (\psi)$, where $x$ is a variable and $\psi$ is a formula.
    \end{enumerate} 

    For formulas of the form $\forall x (\psi)$, we refer to the $\forall$ as a \highlight{universal quantifier}, and to $\psi$ as the \highlight{scope} of that quantifier.
\end{rdefinition}

The fundamental idea behind defining our formulas using parentheses is to eliminate all possible ambiguity, and therefore allow us to prove the unique readability theorem\footnote{Hereafter we will sometimes omit parentheses in formulas if there is no ambiguity.}. To that end, we first require a  result relating the number of left and right parentheses in a formula.

\begin{definition}
    Let $\cou(s)$ be the number of open parentheses in $s$ subtracted by the number of closed parentheses in $s$, where $s$ is a sequence of logical symbols. 
\end{definition}

\begin{lemma}\label{count-def}
    Let $\varphi$ be a formula. If $i \le |\varphi|$, then $\cou(\varphi \restrictto i) \geq 0$ and $\cou(\varphi) = 0$. 
\end{lemma}

\begin{proof}    
    Let $\varphi$ be a formula. We will prove this by induction on $\varphi$. If $\varphi$ is an atomic formula, the result follows since $\cou(\varphi \restrictto i) = 0$ for $i \leq |\varphi|$. Now, consider the case where $\varphi$ is of the form $\neg (\psi)$. Notice that $\cou(\varphi \restrictto 1) = 0$, $\cou(\varphi \restrictto 2) = 1$, and using our inductive hypothesis we have that $\cou(\varphi \restrictto i) - 1 = \cou(\psi \restrictto {i - 2}) \geq 0$ for $2 \leq i \leq |\varphi| - 1$. When $i = |\varphi| - 1$ we have that $\cou(\varphi \restrictto {|\varphi| - 1}) = \cou(\psi \restrictto {|\varphi| - 3}) + 1 = \cou(\psi) + 1 = 1$.  Since the last symbol in $\varphi$ is a closed parenthesis, it follows that $\cou(\varphi) = 0$. The result follows analogously for the case when $\varphi$ is of the form $\forall x (\psi)$. Finally, consider the case where $\varphi$ is of the form $(\psi) \vee (\chi)$. Reasoning along similar lines, we conclude that $\cou(\varphi \restrictto i) \geq 0$ for $i \leq |\psi| + 4$ and that $\cou(\varphi \restrictto {|\psi| + 4}) = 1$. Now, using our inductive hypothesis, we can see that $\cou(\varphi \restrictto i) - 1 = \cou(\chi \restrictto {i -  |\psi|  - 4})$, for $i > |\psi| + 4$. Lastly, we conclude that $\cou(\varphi) = 0$.
\end{proof}

We now have everything we need to state and prove the \highlight{unique readability theorem}, which, as we previously stated, eliminates possible ambiguities when interpreting formulas.

\begin{theorem}
    For a given formula $\varphi$, it is only a formula in a single one of the ways described in \autoref{def-formula}. Furthermore, in each case, the respective $\psi$, or $\psi$ and $\chi$, are uniquely determined. That is, for example, there are no different pairs $\psi$, $\chi$ and $\alpha$, $\beta$ such that $\varphi$ has the form $(\psi) \vee (\chi)$ and $(\alpha) \vee (\beta)$. 
\end{theorem}

\begin{proof}
    Notice that the four possibilities for being a formula in \autoref{def-formula} are exclusive since they denote formulas that begin with different symbols. For example, atomic formulas all begin with a variable, whereas formulas under the second numeral all begin with the symbol $\neg$.
    
    Now let us prove that the way in which $\varphi$ is a formula is uniquely determined. Note that it is clear that atomic formulas and formulas that begin with $\neg$ or $\forall$, are uniquely determined.  Let us consider the case of a formula $\varphi$ that has the form $(\psi) \vee (\chi)$. Towards a contradiction\footnote{Normally proof by contradiction is not valid intuitionistically, but since what we aim to prove is ``not a statement that makes an infinitary claim" (see \cite[pg. 103]{phil}), it is a valid form of inference.}, assume that $\varphi$ also has the form $(\alpha) \vee (\beta)$, where $\psi \neq \alpha$ (it is clear that it cannot be the case that $\psi = \alpha$ and $\chi \neq \beta$). Without loss of generality assume that $|\alpha| < |\psi|$. By \autoref{count-def}, we see that since $\cou(\alpha) = 0$, we have that $\cou(\psi \restrictto {|\alpha|}) = 0$. Notice that the symbol following $\alpha$ in $\varphi$ is $)$, and as such, the symbol at position $|\alpha| + 1$ in $\psi$ is also $)$. This implies that $\cou(\psi \restrictto {|\alpha| + 1}) = \cou(\psi \restrictto {|\alpha|}) - 1 = -1$. This contradicts \autoref{count-def}, and we conclude that $\psi = \alpha$ and $\chi = \beta$.\
\end{proof}

\begin{definition}
    An occurrence of a variable $x$ in a formula $\varphi$ is \highlight{bound} if, and only if,  $x$ appears in a formula $\forall x (\psi)$, that itself appears in $\varphi$. An occurrence is \highlight{free} if, and only if,  it is not bound. $\varphi$ is a \highlight{sentence} if, and only if,  no variable is free in $\varphi$. \highlight{$V(\varphi)$} is the set of variables that have a free occurrence in $\varphi.$
\end{definition}

\begin{definition}
    If $\varphi$ is a formula, a \highlight{universal closure} of $\varphi$ is any sentence of the form $\forall x_1 \forall x_2 \dots \forall x_n \varphi$, where $n \geq 0$.
\end{definition}

To simplify our notation, here we introduce various abbreviations that correspond to the usual definitions of conjunction, implication, equivalence, and existential quantifier:
  
\begin{definition}
    If $\varphi$ and $\psi$ are formulas, we define the following:
    
    \begin{enumerate}
        \item ($\varphi) \wedge (\psi)$ abbreviates $\neg ((\neg(\varphi)) \vee (\neg (\psi))).$
        
        \item $(\varphi) \Rightarrow (\psi)$ abbreviates $(\neg(\varphi)) \vee (\psi)).$
        
        \item $(\varphi) \Leftrightarrow  (\psi)$ abbreviates $((\varphi) \Rightarrow (\psi)) \wedge ((\psi) \Rightarrow (\varphi)).$
        
        \item $\exists x (\varphi)$ abbreviates $\neg (\forall x (\neg (\varphi))).$
    \end{enumerate}
\end{definition}

\subsection{Semantics}

Truth is possibly the most fundamental concept in semantic logic. In common language, trying to define truth leads to antinomies such as the liar's paradox. However, in a formal language, we can rely on specific structural conditions to define the notion of \highlight{true statement}. This development is due to the logician Alfred Tarski who in 1931 managed to define this notion in a formal and non-contradictory way \cite{Tarski}. Broadly speaking, a statement is true if, and only if, it is \highlight{valid} in all \textit{models} and it is false if, and only if, there is no \textit{model} that satisfies it. By the end of this section, we will have defined this new terminology. We start by establishing the notation and the necessary concepts to define the \highlight{validity} of formulas relative to our language.

\begin{definition}\label{def-asig}
    An \highlight{assignment} for a formula $\varphi$ in a non-empty set $A$, is a function $\sigma$ such that $V(\varphi) \subseteq  \dom(\sigma)$ and $\ran(\sigma) \subseteq A.$
\end{definition}

\begin{notation}\label{notation rem}
    $\sigma + (y / a) \coloneqq \sigma\!\restriction_{V(\varphi) \setminus \{y \} } \cup \   \{(y,a)\}$.
\end{notation}

For example, let us consider $\varphi$ as the formula $\forall x (x = y) \vee (w = y)$. In this case, $V(\varphi) = \{ y, w \}$ and if $a,b, a_{y}, a_{w} \in A,$ then $\sigma \coloneqq \{ (y, a_{y}), (w, a_{w})\}$ is an assignment for $\varphi$ in $A.$ Furthermore, according to \autoref{notation rem}, we have that $\sigma + (y/a) = \{ (y, a), (w, a_{w}) \}$ and $(\sigma + (y/a)) + (w / b) = \{ (y, a), (w, b) \}.$ On the other hand, given that in \autoref{def-asig} it was only asked that $ V(\varphi) \subseteq \dom(\sigma),$ then $ \tau \coloneqq \{  (x, a), (y, a_{y}), (z, b), (w, a_{w}) \},$ where $z$ is a variable, is also an assignment for $\varphi$ in $A.$  Next, we recursively define the truth value of a formula with respect to a non-empty set $A$, and an assignment $\sigma$ for said formula in $A$. 

\begin{rdefinition}\label{truth definition}
    
    If $\sigma$ is an assignment for $\varphi$ in a non-empty set $A$, we define $\varphivalue \in \{0, 1\}$ as follows:
    
    \begin{enumerate}
        \item If $\varphi$ is of the form $x = y$: $\varphivalue = 1$ if, and only if,  $\sigma(x) = \sigma(y)$.

        \item If $\varphi$ is of the form $x \in y$: $\varphivalue = 1$ if, and only if,  $\sigma(x) \in \sigma(y)$.
        
        \item If $\varphi$ is of the form $\neg (\psi)$: $\varphivalue = 1 - \psivalue$.
        
        \item If $\varphi$ is of the form $(\psi) \vee (\chi)$: $\varphivalue = 1$ if, and only if,  $\psivalue = 1$ or $\chivalue = 1$.
        
        \item If $\varphi$ is of the form $\forall x (\psi)$: $\varphivalue = 1$ if, and only if,  ${\val}_{A}(\varphi)[\sigma + (x/a)] = 1$ for all $a \in A$.
    \end{enumerate}
    
    If $\varphivalue = 1,$ we say that $\varphi[\sigma]$ is valid in $A$ and we denote it as $A \models \varphi[\sigma]$.
\end{rdefinition}

From this definition, we obtain an algorithm that assigns, to each formula $\varphi$ and each assignment $\sigma$ for $\varphi$ in $A$, a truth value, $\varphivalue \in \{ 0, 1\}$. By appealing to the unique readability theorem it is clear that this algorithm always assigns the same value to a given formula and assignment. That is, if we think informally of $\val_{A}(\varphi)[\sigma]$ as a function, then:

\begin{lemma}
    $\varphivalue$ is well-defined.
\end{lemma}

As expected, the truth value of a function with a given assignment depends only on the assignment of the variables that appear in the formula. In other words:

\begin{lemma}\label{value-only}
    If $\tau$ and $\sigma$ are assignments for $\varphi$ in a non-empty set $A$ and $\tau \restrictto{V(\varphi)} = \sigma \restrictto{V(\varphi)},$ then  $\varphivalue = \val_{A}(\varphi)[\tau].$
\end{lemma}

\begin{proof}
    Let $\varphi$ be a formula and let $\tau$ and $\sigma$ be assignments for $\varphi$ in a non-empty set $A$, such that $\tau\!\restriction_{V(\varphi)} \ = \sigma\!\restriction_{V(\varphi)}$. We will prove this lemma by induction on $\varphi.$
    
    Assume $\varphi$ is an atomic formula of the form $x \in y$. This implies that  $x, y \in V(\varphi)$ and therefore that $\sigma(x) = \tau(x)$ and $\sigma(y) = \tau (y)$. It follows that, $$\varphivalue = 1 \text{ iff }  \sigma(x) \in \sigma(y) \text{ iff }  \tau(x) \in \tau(y) \text{ iff }  \varphivaluep = 1.$$ Thus, $\varphivalue = \val_{A}(\varphi)[\tau].$ In the case where $\varphi$ is of the form $x = y$ the result follows analogously.
    
    Notice that for any subset $U$ of $V(\varphi)$, we have that $\sigma\!\restriction_{U} \ = \tau\!\restriction_{U}$. Now suppose that $\varphi$ is the form $(\psi) \vee (\chi)$. It is clear that $V(\psi), V(\chi) \subseteq V(\varphi)$ and therefore it then follows that, 
    \begin{equation*}
        \begin{split}
            \varphivalue = 1 & \text{ iff }  \val_{A}(\psi)[\sigma] = 1 \text{ or } \val_{A}(\chi)[\sigma] = 1 \\
            & \text{ iff }  \val_{A}(\psi)[\tau] = 1 \text{ or } \val_{A}(\chi)[\tau] = 1 \\
            & \text{ iff }  \varphivaluep = 1. \\
        \end{split}
    \end{equation*}
    
    Thus $\varphivalue = \val_{A}(\varphi)[\tau]$. The result would follow analogously if $\varphi$ was of the form $\neg (\psi)$ or $\forall x (\psi)$.
\end{proof}

\subsection{Propositional tautologies and logical validity}\label{tautologies}

In everyday language, a tautology is an obvious or redundant assertion that is true under every possible interpretation. Similarly, in the context of logic, a propositional tautology is a formula that is ascertained to always be true solely by virtue of the meaning of the propositional connectives it is composed of, without reference to the meaning of the symbols $\forall$, $\in$, or $=$. For example, the formula $\varphi \wedge \neg \varphi$ is a propositional tautology since its truth value is always $1$ independent of the content of $\varphi$. On the other hand, $\forall x (\varphi(x)) \Rightarrow \forall y (\varphi (y))$ is not one since its truth value depends on the way we interpret the symbol $\forall$. Next, we give a recursive definition of a propositional tautology:

\begin{rdefinition}
    A \highlight{basic formula} is a formula that is not of the form $\neg (\psi)$ or $(\psi) \vee (\chi)$. A \highlight{truth assignment} for $\varphi$ is a function that maps the basic formulas that appear in $\varphi$ into $\{0, 1\}$. Given such a truth assignment, $v$, we extend it into $\overline{v} \in \{0, 1\}$ as follows:
    
    \begin{enumerate}

        \item If $\varphi$ is of the form $\neg (\psi)$: $\overline{v}(\varphi)$ = $1 - \overline{v}(\psi)$.
        
        \item If $\varphi$ is of the form $(\psi) \vee (\chi)$: $\overline{v}(\varphi) = 1$ if, and only if, $\overline{v}(\psi) = 1$ or $\overline{v}(\chi) = 1$.
    \end{enumerate}
    
    We say that $\varphi$ is a \highlight{propositional tautology} if, and only if $\overline{v}(\varphi) = 1$ for all truth assignments $v$.
\end{rdefinition}

It is not puzzling, given this definition, that propositional tautologies will always be valid independent of the set and assignment used to calculate their truth value. Formulas that have this property are precisely the ones that Tarski conceived as logical truths and we define them in the following way:

\begin{definition}
    The formula $\varphi$ is \highlight{logically valid} if, and only if, for each non-empty set $A, $ $A \models \varphi[\sigma]$ for all assignments $\sigma$ for $\varphi$ in $A.$    
\end{definition}

For example, if $\varphi$ is $\forall x (x = x),$ then $\varphi$ is logically valid. Indeed, consider a non-empty set $A,$ an assignment for $\varphi$ in $A$ and $a \in A.$ Note that, by \autoref{truth definition}, $\val_{A}(x = x)[\sigma + (x / a)] = 1 \text{ iff }  (\sigma + (x / a))(x) = (\sigma + (x / a))(x) \text{ iff }  a = a.$ Therefore, for all $a \in A,$ we have that $\val_{A}(x = x)[\sigma + (x/a)] = 1;$ that is, $A \models \varphi[\sigma]$. Therefore $\forall x (x = x)$ is logically valid. And as mentioned above, propositional tautologies are also cases of logically valid formulas:

\begin{lemma}\label{tautologies valid}
    Every propositional tautology is logically valid. 
\end{lemma}

\begin{proof}
    Let $\varphi$ be a propositional tautology, $A$ a non-empty set, and $\sigma$ an assignment for $\varphi$ in $A$. It is clear that we can construct a truth assignment $v$ such that for all basic formulas $\psi$ in $\varphi$, we have that $v(\psi) = \psivalue$. Notice that the way in which we extend $v$ into $\overline{v}$ matches exactly with the way that $\psivalue$ is constructed from its composite parts. As a result, we have that $\overline{v}(\psi) = \psivalue$ for all subformulas $\psi$ in $\varphi$. In particular, we have that $\overline{v}(\varphi) = \varphivalue$, and since $\overline{v}(\varphi) = 1$ by virtue of it being a propositional tautology, we have that $A \models \varphi [\sigma]$. We conclude that $\varphi$ is logically valid.
\end{proof}

\subsection{Some relations between syntactic and semantic notions}

In this section, we define some syntactic notions that will be necessary later and we also prove various semantic results that relate to these notions.

\begin{definition}
    If $\varphi$ is a formula and $x$ and $y$ are variables, then $\varphi(x \, {\leadsto} \,  y)$ is the formula that results from $\varphi$ by replacing all free occurrences of $x$ by $y$. 
\end{definition}

It is clear that $\varphi(x \, {\leadsto} \,  y)$ is a formula. Informally, we can see that $ \varphi(x \, {\leadsto} \,  y) $ expresses about $ y $ what $ \varphi $ does about $x.$ Now, let us consider an example: if $ \varphi $ is the formula $ \forall x ((x \in y) \vee (y \in x)) $ and $ z $ is a variable, then $ \varphi (y \, {\leadsto} \,  z) $ is the formula $ \forall x ((x \in z) \vee (z \in x)). $ Also, since no occurrence of $ x $ appears free in $ \varphi, $ then $\varphi (x \, {\leadsto} \,  z) $ is $ \varphi. $

\begin{definition}
    A variable $y$ is \highlight{free for} a variable $x$ in a formula $\varphi$ if, and only if, no free occurrence of $x$ appears in any subformula $\forall y (\psi)$, that itself appears in $\varphi.$
\end{definition}

For example, consider $ \varphi $ as the formula $ \forall y (y = z) \vee x \in y. $ Since no free occurrence of $ x $ appears under the scope of a quantifier $ \forall y , $ it follows that $ y $ is free for $ x $ in $ \varphi. $ Also, if $ \psi $ is the formula $ \forall x (x \in y), $ then it is clear that $ x $ is not free for $y $ in $ \psi. $

The next lemma, which we call the \highlight{substitution lemma}, is a result that allows us to calculate the truth value of $ \varphi (x \, {\leadsto} \,  y) [\sigma] $ in terms of the truth value of $ \varphi $ and $ \sigma (y)$:

\begin{lemma}\label{sustitution-lemma}
    Let $A$ be a non-empty set, $\varphi$ a formula, $y$ a variable, and $\sigma$ a assignment for $\varphi$ in $A$ such that $y \in \dom(\sigma).$ If $y$ is free for $x$ in $\varphi,$ then $$ \val_{A}(\varphi(x \, {\leadsto} \,  y))[\sigma] = \val_{A}(\varphi)[\sigma + (x / \sigma(y))]. $$
\end{lemma}

\begin{proof}    
   Notice that, if $x$ does not appear free in $\varphi,$ then $\varphi(x \, {\leadsto} \,  y)$ is $\varphi$, and the value assigned by $\sigma$ to $x$ is irrelevant and therefore the result is clear. Therefore, suppose that $x$ appears free in $\varphi$ and we will apply induction on $\varphi.$ The base case is when $\varphi$ is atomic. Therefore, $\varphi$ has one of the following forms: $x \in w, w \in x, x = w $ or $w = x,$ where $w$ is a variable. Let us  verify, for example, the case when  $\varphi$ is of the form $x \in w$; the other cases are completely analogous. Notice that, in this case, $\varphi(x \, {\leadsto} \,  y)$ is $y \in w$, and further, $(\sigma + (x / \sigma(y)))(x) = \sigma(y)$ and $(\sigma + (x / \sigma(y)))(w) = \sigma(w).$ Therefore, we have that:
    \begin{equation*}
        \begin{split}
            \val_{A}(\varphi(x \, {\leadsto} \,  y))[\sigma] = 1 & \text{ iff }  \sigma(y) \in \sigma(w)\\
            & \text{ iff } (\sigma + (x / \sigma(y)))(x) \in (\sigma + (x / \sigma(y)))(w)\\
            & \text{ iff } \val_{A}(\varphi)[\sigma + (x / \sigma(y))] = 1. 
        \end{split}
    \end{equation*}  
    
    Now, suppose that $\varphi$ is of the form $\neg (\psi).$ It is clear that $\varphi(x \, {\leadsto} \,  y) = \neg \psi(x \, {\leadsto} \,  y),$ hence by induction hypothesis we have that: 
    \begin{equation*}
        \begin{split}
            \val_{A}(\varphi(x \, {\leadsto} \,  y))[\sigma] = 1 & \text{ iff }  \val_{A}(\neg(\psi(x \, {\leadsto} \,  y)))[\sigma] = 1\\
            & \text{ iff }  \val_{A}(\psi(x \, {\leadsto} \,  y))[\sigma] = 0\\
            & \text{ iff }  \val_{A}(\psi)[\sigma + (x / \sigma(y))] = 0\\
            & \text{ iff }  \val_{A}(\neg (\psi))[\sigma + (x/\sigma(y))] = 1\\
            &\text{ iff }  \val_{A}(\varphi)[\sigma + (x / \sigma(y))] = 1.
        \end{split}
    \end{equation*}
    Now, if $\varphi$ is of the form $(\psi) \vee (\chi),$ it is clear that $\varphi(x \, {\leadsto} \,  y) $ is $ (\psi(x \, {\leadsto} \,  y)) \vee (\chi(x \, {\leadsto} \,  y))$ and therefore, the reasoning is similar to the above.
    
    Finally, suppose that $\varphi$ is of the form $\forall z (\psi).$ Since $x$ is free in $\phi$ then $x$ must not be equal to $z$. Similarly, since $y$ is free for $x$ in $\varphi,$ it also follows that $y$ is not equal to $z$. So we have that $\varphi(x \, {\leadsto} \,  y) $ is $ \forall z(\psi(x \, {\leadsto} \,  y)).$ Applying the induction hypothesis, it follows that:
    \begin{equation*}
        \begin{split}
            \val_{A}(\varphi(x \, {\leadsto} \,  y))[\sigma] = 1 &\text{ iff }  \val_{A}(\forall z (\psi(x \, {\leadsto} \,  y)))[\sigma] = 1\\
            & \text{ iff }  (\val_{A}(\psi(x \, {\leadsto} \,  y))[\sigma + (z/ a)] = 1) \text{ for all} \ a \in A\\
            & \text{ iff }  (\val_{A}(\psi)[\sigma + (x / \sigma(y)) + (z / a)] = 1) \text{ for all} \  a \in A\\
            & \text{ iff }  \val_{A}(\forall z (\psi))[\sigma + (x / \sigma(y))] = 1\\
            & \text{ iff }  \val_{A}(\varphi)[\sigma + (x / \sigma(y))] = 1.
        \end{split}
    \end{equation*}
    Finally, $\val_{A}(\varphi(x \, {\leadsto} \,  y))[\sigma] = \val_{A}(\varphi)[\sigma + (x / \sigma(y))]. $
\end{proof}

In \autoref{tautologies} we saw an example of a logically valid formula, and we saw that tautologies themselves are also logically valid. The following result shows us more interesting formulas that are also logically valid, and which will subsequently appear in the axioms of first-order logic.

\begin{corollary}\label{axioms-valid-cor}
    The following formulas are logically valid:
    
    \begin{enumerate}
         \item $\varphi \Rightarrow  \forall x (\varphi)$, where $x$ is not free in $\varphi.$
        
        \item $\forall x (\varphi \Rightarrow \psi) \Rightarrow (\forall x (\varphi) \Rightarrow \forall x (\psi)).$
        
        \item $\forall x (\varphi) \Rightarrow \varphi(x \, {\leadsto} \,  y),$ where $y$ is free for $x$ in $\varphi.$ 
    \end{enumerate}
\end{corollary}

\begin{proof}
    Suppose that $\varphi$ and $\psi$ are formulas. Fix a non-empty set $A.$ Then,  
    \begin{enumerate}
        \item Define $\chi$ as $\varphi \Rightarrow \forall x (\varphi)$. Assume that $x$ is not free in $\varphi,$ $\sigma$ is an assignment for $\chi$ in $A$, and that $A \models \varphi[\sigma].$ Fix $a \in A.$ Note that, since $x$ is not free in $\varphi,$ $\sigma \restrictto {V(\varphi)} = \tau \restrictto {V(\varphi)},$ where $\tau  \coloneqq \sigma + (x / a).$  Therefore, by \autoref{value-only}, we have that $$1 = \val_{A}(\varphi)[\sigma] = \val_{A}(\varphi)[\tau] = \val_{A}(\varphi)[\sigma + (x / a)].$$ Thus, for every $a \in A$ we have that $\val_{A}(\varphi)[\sigma + (x / a)] = 1;$ that is, $$\val_{A}(\forall x (\varphi))[\sigma] = 1.$$ Finally, $A \models (\varphi \Rightarrow \forall x \varphi)[\sigma].$
        
        \item Define $\chi$ as  $\forall x (\varphi \Rightarrow \psi) \Rightarrow (\forall x (\varphi) \Rightarrow \forall x (\psi)).$ Assume that $\sigma$ is an assignment for $\chi$ in $A$ and  that $A \models  (\forall x (\varphi \Rightarrow \psi))[\sigma]$   and $A \models (\forall x (\varphi))[\sigma]. $ Fix $a \in A.$  By definition, $A \models (\varphi \Rightarrow \psi)[\sigma + (x/a)]$ and $A \models \varphi[\sigma + (x / a)].$ Therefore, $A \models \psi[\sigma + (x / a)].$ Then, for every $a \in A$, we have that $\val_{A}(\psi)[\sigma + (x / a)] = 1.$  Finally, $A \models (\forall x)(\psi)[\sigma],$ and hence $A \models  \forall x (\varphi \Rightarrow \psi) \Rightarrow (\forall x (\varphi) \Rightarrow \forall x (\psi))[\sigma]. $
        
        \item Define $\chi$ as $\forall x (\varphi) \Rightarrow \varphi(x \, {\leadsto} \,  y).$ Assume that $y$ is free for $x$ in $\varphi,$ $\sigma$ is an assignment for $\chi$ in $A$, and that $A \models (\forall x(\varphi))[\sigma].$ There are two cases. First, suppose that $y$ is free in $\chi.$ Therefore, $y \in \dom(\sigma)$ and so $\sigma(y) \in A.$ Hence, in particular, $A \models \varphi[\sigma + (x / (\sigma(y)))]$ and note that  we are under the hypotheses of \autoref{sustitution-lemma}, by virtue of which, we have that $$\val_{A}(\varphi(x \, {\leadsto} \,  y))[\sigma] = \val_{A}(\varphi)[\sigma + (x / \sigma(y))] = 1.$$ Thus, $A \models \varphi(x \, {\leadsto} \,  y)[\sigma].$ 
        
        Now, if $ y $ is not free in $ \chi, $ then it is not free in $ \forall x (\varphi) $ and in $ \varphi (x \, {\leadsto} \,  y). $ Towards a contradiction, suppose that $ x $ appears free in $ \varphi. $ Since $ y $ is free for $ x $ in $ \varphi, $ then no free occurrence of $ x $ is under the scope of a $ \forall y $ quantifier, and therefore there must be a free occurrence of $ y $ in $ \varphi (x \, {\leadsto} \,  y), $ which contradicts that $ y $ is not free in $ \varphi (x \, {\leadsto} \,  y). $ It follows that $ x $ is not free in $ \varphi$, and therefore $ \varphi (x \, {\leadsto} \,  y) $ is $ \varphi, $ from which the result follows.
    \end{enumerate}
\end{proof}

\subsection{Proof theory and soundness}

This section introduces the last concepts necessary to build our logical system: logical axioms, inference rules, and formal proofs. We finish by proving the soundness theorem, a pivotal result that ties together all of these concepts.  We begin by introducing the first-order logical axioms, which seek to formalize elementary intuitions about the structure of validity and equality:

\begin{definition}
     If $\varphi$ and $\psi$ are arbitrary formulas, then a \highlight{logical axiom} is a universal closure of one the types of formulas below:
    \begin{enumerate}
     
        \item Propositional tautologies.
        
        \item $\varphi \Rightarrow \forall x (\varphi),$ where $x$ is not free in $\varphi.$
        
        \item $\forall x (\varphi \Rightarrow \psi) \Rightarrow (\forall x (\varphi) \Rightarrow \forall x (\psi))$
        
        \item $\forall x (\varphi) \Rightarrow \varphi(x \, {\leadsto} \,  y),$ where $y$ is free for $x$ in $\varphi.$ 

        \item $x = x.$
        
        \item $x = y \Leftrightarrow  y = x.$
        
        \item ($x = y \wedge y = z)  \Rightarrow x = z.$
        
        \item $(w = x \wedge y = z) \Rightarrow (w \in y \Leftrightarrow  x \in z).$
    \end{enumerate}
\end{definition}

For example, axioms of the type $(5), (6)$, and $(7)$ state an intuitively clear fact about equality: it is an equivalence relation. There are two main reasons for choosing these axioms: on the one hand, they are sufficient to prove the \emph{completeness theorem} (see \cite[II.12]{Kunen}), and on the other, the axioms are logically valid, which is necessary if we want to prove the \emph{soundness theorem}. That these axioms are intended to formalize elementary intuitions can be interpreted as saying that they are logically valid. This easily follows from results we have previously proved:

\begin{theorem}\label{axioms logic valid}
    All the logical axioms are logically valid.
\end{theorem}

\begin{proof}
    Notice that, by \autoref{truth definition} is clear that the axioms of type $(5)-(8)$ Furthermore, by \autoref{tautologies valid}, the axioms of type $(1)$ are logically valid. Finally, to see that the axioms of type $(2), (3)$ and $(4)$ are logically valid refer to the proof of \autoref{axioms-valid-cor}. 
\end{proof}

Another of the notions necessary to be able to complete the construction of our logical system is the notion of \highlight{inferences rules}. An \highlight{inference rule} is a reasoning that takes some \highlight{premises} and after analyzing its syntax yields as a result a \highlight{conclusion}. The only inference rule in our logical system is ``\highlight{Modus Ponens}'', which takes two premises: one of the form ``$\varphi \Rightarrow \psi$'' and the other one of the form ``$\varphi$'', and returns ``$\psi$''. Or, in a symbolic form: $$\frac{\varphi, \varphi \Rightarrow \psi}{\psi}$$ It is clear, by \autoref{truth definition}, that Modus Ponens is valid in the following sense: if it takes valid premises then its conclusion is also valid. Having introduced Modus Ponens as the only inference rule, we are able to introduce the definition of \highlight{formal proof}: 

\begin{definition}
     Suppose that $\Sigma$ is a set of sentences. A \highlight{formal proof} from $\Sigma$ is a finite, non-empty sequence of sentences $\varphi_{0}, \dots, \varphi_{n}$, such that for each $i$ either $\varphi_{i} \in \Sigma$ or $\varphi_{i}$ is a logical axiom, or for some $j, k < i, \ \varphi_{i}$ follows from $\varphi_{j}, \varphi_{k}$ by Modus Ponens. This sequence is a formal proof of its last sentence, $\varphi_{n}.$  
\end{definition}

\begin{definition}
     If $\Sigma$ is a set of sentences and $\varphi$ is a sentence, then $\Sigma \vdash \varphi$ if, and only if,  there exists a formal proof of $\varphi$ from $\Sigma.$ 
\end{definition}

Now, we generalize the idea of validity to sets of sentences:

\begin{definition}
    If $A$ is a non-empty set and $\Sigma$ is a set of sentences, then $A \models \Sigma$ if, and only if,  $A \models \varphi$ for each $\varphi \in \Sigma$. 
\end{definition}

This generalization allows us to naturally define a semantic relation similar to the formal proof relationship that can occur between a set of sentences $\Sigma$ and a sentence $\varphi$. 

\begin{definition}
     If $\Sigma$ is a set of sentences and $\varphi$ is a sentence, then $\Sigma \models \varphi$ if, and only if,  $A \models \varphi$ for every non-empty set $A$ such that $A \models \Sigma.$  
\end{definition}

So, given a set of sentences $\Sigma$ and a sentence $\varphi,$ we have defined two relations that can occur between them: a semantic one, namely $\Sigma \models \varphi$, and a syntactic one, namely $\Sigma \vdash \varphi.$ It is then natural to wonder about the relation that exists between the semantic and the syntactic notion. The first relation that we will be able to prove is known as the \highlight{soundness theorem} which states that $\Sigma \vdash \varphi \text{ implies } \Sigma \models \varphi$ or, in other words, that our logical system does not prove statements to be false:

\begin{theorem}\label{soundness}
    If $\Sigma \vdash  \varphi,$ then $\Sigma \models \varphi.$ 
\end{theorem}

\begin{proof}
    Suppose that $\Sigma \vdash \varphi$ and fix a non-empty set $A$ such that $A \models \Sigma.$ Let $\varphi_{0}, \varphi_{1}, \dots, \varphi_{n}$ be a formal proof of $\varphi$ from $\Sigma.$ To show that $A \models \varphi,$ we will apply induction on $i \in \mathbb{N}$ to see that $A \models \varphi_{i}.$ Notice that if $\varphi_{i} \in \Sigma$ the result is clear. Also, if $\varphi_{i}$ is a logical axiom, the result follows from  \autoref{axioms logic valid}. Now, if Modus Ponens is used, there exists $j, k < i$ such that $\varphi_{k}$ is $(\varphi_{j} \Rightarrow \varphi_{i})$ and $\varphi_{i}$ is the result of applying Modus Ponens to $\varphi_{k}$ and $\varphi_{j}.$ Since $j, k < i,$ by hypothesis $A \models \varphi_{k}$ and $A \models \varphi_{k};$ that is, $A \models (\varphi_{j} \Rightarrow \varphi_{i})$ and $A \models \varphi_{j}.$ Finally, it is clear from \autoref{truth definition} that $A \models \varphi_{i}.$
\end{proof}

The converse, which is much less obvious, is also true (in $\mathrm{ZFC}$), and is a result of Kurt Gödel known as the \highlight{completeness theorem}. We need not prove this theorem for our purposes, but its proof can be found in \cite[Sec~II.10]{Kunen}. 

\autoref{soundness} determines the end of the construction of our logical system and also, as we mentioned before, indicates that we cannot prove false statements. This gives us confidence in the formal proofs that we carry out in the system, and this result can be thought of as a direct consequence of the confidence we have in the type of reasoning accepted in the metatheory. Now, if we summarize briefly what has been constructed in this chapter, we can say that we have defined, using metatheoretical reasoning, a \highlight{sound logical system}, which is roughly composed of the following elements: an alphabet, formulas, syntax and semantics, an inference rule, logical axioms, and formal proofs. Now, all that remains is to introduce the formal theory of sets based on this system.

With this in mind, we note that all of the axioms of set theory can be written using our alphabet $\mathcal{A}$ (see \cite{Kunen}), and therefore we can define in our logical system a set of sentences that correspond to the usual axioms of \emph{set theory}. This set of sentences, together with the tools of our logical system, is called  formal set theory (abbreviated simply as $\thzfc$)—a theory in which the vast majority of known mathematics can be formalized.

\section{Attempting to Answer the Question}

The work before us laid bare the dependencies and structure of the metatheoretical reasoning necessary to construct a sound system of first-order logic. We now turn to this construction as our object of study, and attempt to answer the question we originally posed: \emph{which came first, set theory or logic?}

At first glance we can say, perhaps with some trepidation, that logic came first! It is clear that the set metatheory that we used in creating our logical system is distinct and substantially less powerful (if for no other reason than it being limited to finite sets) than $\thzfc$. In other words, we were able to construct a logical system capable of hosting $\thzfc$, by using something that is not quite $\thzfc$.

On the other hand, set metatheory \emph{looks} quite a bit more like $\thzfc$ than our metatheoretical logic looks like our formal logical system. Informally speaking, set metatheory is a fragment of $\thzfc$ whose universe of discourse is the hereditarily finite sets (see \cite[Def. I.14.13]{Kunen}). Dissimilarly, our metatheoretical logic is not a less powerful form of first-order logic, but a different form of reasoning altogether. With this characterization in mind, we can argue that by using our primary intuitions (intuitionistic logic) we were able to create a basic set theory. This line of reasoning leads us to answer our initial question by saying that (a fragment of) set theory came first!

Perhaps not surprisingly, there may be some who disagree with both of these answers. Our original question leaves a lot to interpretation; it does not, for example, specify whether the logic it asks about is first-order logic.  This leads to our third answer: that (metatheoretical) logic came first.

Having reached the very beginning of our construction, and consequently, of more possible answers, we may now breathe a false sigh of relief. There is however another way of viewing our metatheoretical logic that affects how we would answer the question. As noted earlier, within our metatheoretical logic lie interspersed some very basic set-theoretical notions. In its exposition, we talk about comparing quantities, finite algorithms, etc., and explicitly note that it is ``bootstrapped from our intuitive notions of [...] finite sets". It is not that these notions come \emph{before} metatheoretical logic, but rather more like they are intertwined and exist concurrently as part of our innate rational faculties. Per a similar argument as in the preceding paragraph, we could view these notions as being worthy of the name ``set theory''. This viewpoint would have us answer the question with ``neither'', or ``both'', or ``the metatheory''.

\section*{Funding}

This work was supported by the Austrian Science Fund (FWF) [P33895 to AFUZ].

\section*{Acknowledgements}

The authors would like to thank Professor Diego A. Mejía of Shizuoka University for proofreading this article and giving valuable suggestions. The authors would also like to thank Professor Carlos Mario Parra Londoño, of the Universidad Nacional de Colombia Sede Medellín, for introducing the authors to the bibliography used in this paper and for helpful discussions.

{\small
\bibliography{bibli}

\begin{thebibliography}{}

\bibitem[Kleene, 1971]{Kleene}
Kleene, S.~C. (1971).
\newblock {\em Introduction to Metamathematics}.
\newblock Wolters-Noordhoff Publishers and North Holland Publishing Company.

\bibitem[Kunen, 2012]{Kunen}
Kunen, K. (2012).
\newblock {\em The Foundation of Mathematics}.
\newblock College Publications, London.

\bibitem[Tarski, 1931]{Tarski}
Tarski, A. (1931).
\newblock Sur les ensembles définissables de nombres réels. i.
\newblock {\em Fundamenta Mathematicae}, 17:210--239.

\bibitem[Velleman and George, 2002]{phil}
Velleman, D.~J. and George, A. (2002).
\newblock {\em Philosophies of Mathematics}.
\newblock BlackWell Publishers.

\end{thebibliography}
\bibliographystyle{apalike}
}

\end{document}